\pdfoutput=1 

\documentclass[final]{dmtcs-episciences}

\usepackage[utf8]{inputenc}
\usepackage{subfigure}

\usepackage[round]{natbib}

\usepackage{amsmath,amssymb,amsthm}
\usepackage{graphicx}
\usepackage{hyperref}

\title{On the Heapability of finite partial orders}

\usepackage{tikz}
\usetikzlibrary{positioning}
\usetikzlibrary{arrows}
\usepackage{pifont}
 
\usepackage{color}
\definecolor{pastelred}{rgb}{1.0, 0.41, 0.38}
\definecolor{pistachio}{rgb}{0.58, 0.77, 0.45}

\usepackage{graphicx}

\usepackage{amsthm,amsmath}

\newtheorem{theorem}{Theorem}
\newtheorem{lemma}{Lemma}
\newtheorem{corollary}{Corollary}
\newtheorem{proposition}{Proposition}
\newtheorem{observation}{Observation}
\newtheorem{claim}{Claim}

\newtheorem{definition}{Definition}
\newtheorem{conjecture}{Conjecture}
\newtheorem{open}{Open problem}

\title{On the Heapability of Finite Partial Orders}

\author{J\'anos Balogh\thanks{Supported by the European Union, co-financed by the European
Social Fund (EFOP-3.6.3-VEKOP-16-2017-00002).} \affiliationmark{1}
\and
Cosmin Bonchi\c{s}\thanks{Supported in part by a grant of the Romanian National Authority for Scientific Research and Innovation, CNCS-UEFISCDI, project no. PN-II-RU-TE-2014-4-0270 - FraDys: "Theoretical and Numerical Analysis of Fractional-Order Dynamical Systems and Applications".}  \affiliationmark{2,3}  \and  Diana Dini\c{s}\affiliationmark{2,3}
\and Gabriel Istrate\thanks{Corresponding author: Gabriel Istrate. G. I. and D.D. were supported in part by a grant of Ministry of Research and Innovation, CNCS - UEFISCDI, project number
PN-III-P4-ID-PCE-2016-0842, within PNCDI III,  ATCO: "Advanced techniques in optimization and computational complexity". } \affiliationmark{2,3}
\and 
Ioan Todinca\thanks{Partially supported by the ANR (Agence Nationale pour la Recherche) project GraphEn, ANR-15-CE40-0009.} \affiliationmark{4}
}

\affiliation{
 Institute of Informatics, 
 University of Szeged, Hungary \\
 Department of Computer Science, West University of Timi\c{s}oara, Timi\c{s}oara, Romania.\\
e-Austria Research Institute, Timi\c{s}oara, Romania \\
LIFO, Universit\'e d'Orl\'eans and INSA Centre-Val de Loire, France
  }

\received{2018-5-16}
\revised{2020-3-25}

\accepted{2020-5-12}

\begin{document}
\publicationdetails{22}{2020}{1}{17}{4510}

\maketitle

\begin{abstract}
Define a sequence of elements from a partially ordered set to be \emph{heapable} if it can be successively inserted as the leaves of a binary tree, not necessarily complete, such that children nodes are always greater or equal than parent nodes. This is a natural binary analog of the notion of a chain in a poset and an easy extension of a definition for integers due to Byers et al. (2011). A set of elements is called heapable if some permutation of its elements is. 

  We investigate the partitioning of sequences from a poset into a minimal number of
  heapable subsequences. We give an extension of Fulkerson's proof of Dilworth's theorem to decomposition into heapable subsequences which yields as a byproduct a flow-based algorithm for computing such a minimal decomposition. 
  
  On the other hand, for sets and sequences of intervals and for trapezoid partial orders we prove that such minimal decompositions can be computed via simple greedy-type algorithms. 
  
  Second, while the complexity of computing a maximal heapable subsequence of integers is still open, we show that this problem has a polynomial time algorithm for sequences of \emph{intervals}. 
  
  The paper concludes with a couple of open problems related to the analog of the Ulam-Hammersley problem for sets and sequences of random
  intervals. 
\end{abstract}

\section{Introduction}

The {\em longest increasing subsequence} is a classical problem in
combinatorics and algorithmics. Decompositions of (random) permutations into a minimal number of increasing sequences have been studied in the context of the famous \emph{Ulam-Hammesley problem}.  This problem concerns the asymptotic scaling behavior of the length of the longest increasing subsequence (LIS) of a random permutation, and is a problem with deep connections to statistical physics and random matrix theory (for a very readable introduction see \cite{romik2015surprising}).

An interesting variation on the concept of increasing sequence was introduced by Byers et al. \cite{byers2011heapable}: call sequence of integers $A=a_{1},a_{2},\ldots, a_{n}$
\emph{heapable} if numbers $a_{1},a_{2},\ldots, a_{n}$ can be successively
inserted as leaves into a heap-ordered binary tree (not necessarily
complete). Heapability is a "weakly increasing pattern" for integer sequences. The study of patterns in (random) integer sequences (\cite{dmtcs:5011}), has been, of course, quite popular lately, some of the investigated problems even involving increasingly labelled trees (\cite{dmtcs:5691}) or heaps 
(\cite{dmtcs:5796}). 

Heapability was further investigated in \cite{istrate2015heapable} (and,
independently, in \cite{heapability-thesis}). In particular, a subgroup of the
authors of the present papers showed that for permutations one can compute in
polynomial time a minimal decomposition into heapable subsequences, and investigated a problem that can be viewed as an analog of the Ulam-Hammersley problem for heapable sequences (see also
subsequent work in \cite{istrate2016heapability,basdevant2016hammersley,basdevant2017almost,hammersley-fl,chandrasekaran2019maximum}, that extends/confirms some of the conjectures of \cite{istrate2015heapable}). Furthermore, as shown in \cite{istrate2016heapability} one can meaningfully study the analogs of heapability and the Ulam-Hammesley problem in the context of \emph{partial orders}. 

This paper started as a conversation on the heapability of {\em sequences of
  intervals} during a joint Timi\c{s}oara Szeged seminar on theoretical computer
science in November 2015. Its main purpose is \emph{to offer a different perspective on
the concept of heapability, by relating it to well-known results in combinatorics, such as the 
classical theorems of Dilworth and K\H{o}nig-Egerv\'ary. }

Specifically, we prove that the number of classes in a minimal decomposition of a sequence of elements from a poset $P$ into "heapable subsequences" can be obtained as the size of a minimum vertex cover in a certain bipartite graph whose construction directly generalizes the one employed in one classical proof of Dilworth's theorem. As a byproduct, we obtain an efficient algorithm based on network flows for computing such a minimal decomposition. 

This result, together with the ones from \cite{istrate2015heapable} (where such a minimal decomposition
was computed, for integer sequences, via a direct, greedy algorithm) raise the question whether such greedy algorithms exist for other posets except the set of integers. We answer this question in the affirmative way by showing that the result of \cite{istrate2015heapable} is extendible to {\it sets and sequences of intervals}, ordered by the natural partial order. 

The structure of the paper is as follows: in Section \ref{sec:prelim} we review
the main concepts and technical results we will be concerned with. Then, in
Section \ref{sec:main} we prove a result (Theorem~\ref{main:thm}) which shows that part of Fulkerson's  proof of Dilworth's theorem (\cite{fulkerson1956note}) can be extended to characterizing partitions into a minimal number of heapable subsequences. This result provides, as a byproduct, a flow-based algorithm for the computation of such a minimal partition. 

We then investigate, in Section \ref{sec:interval}, the heapability of  sequences of intervals, showing (Theorem~\ref{intervals}) that a greedy-type algorithm computes such a minimal partition. In Section \ref{sec:sets} we show (Theorem~\ref{sets}) that this result  extends to (unordered) {\it sets} of intervals as well. We then give  in Section~\ref{trapez} a second, incomparable, extension (Theorem~\ref{thm:boxes}) from sequences of intervals  to sequences from so-called \emph{trapezoid} partial orders. 

In Section~\ref{maxheap} we investigate the problem of computing a maximal  heapable subsequence of a sequence of elements from a poset.  For sequences of integers the complexity of this problem is open (\cite{byers2011heapable}). 
We show (Theorem~\ref{max}) that it has polynomial time algorithms for sequences of intervals. 

We conclude in Section \ref{sec:open} with some open questions raised by our results. 

\section{Preliminaries} 
\label{sec:prelim}

We will assume knowledge of standard graph-theoretic notions. In particular,
given integer $k\geq 1$, rooted tree $T$ is  {\it $k$-ary} if every node has at
most $k$ children. Given a graph $G=(V,E)$, we will denote by $vc(G)$
the size of the minimum vertex cover of $G$. 

\begin{definition} 
Let $k\geq 1$. A sequence of integers $A=(a_{1},a_{2},\ldots,a_{n})$ is called
{\em $k$-heapable} (or a \emph{
  $k$-ary chain}) if there exists a $k$-ary tree $T$ with $n$ nodes labeled by $a_{1},a_{2},\ldots,a_{n}$, such that for any two nodes labeled
$a_{i},a_{j}$, if $a_j$ is a descendant of $a_i$ in $T$ then $i<j$ holds for their
indices and $a_{i}<a_{j}$. 
\end{definition}

We will be concerned with finite partially ordered sets (posets) only. Sequences
of elements from a poset  naturally embed into this framework by associating, to
every sequence $A=(a_{1},\ldots, a_{n})$ the poset $Q_{A}=\{(i,a_{i}):1\leq i\leq n\}$ with partial order $(i,a_{i})\leq(j,a_{j})$ if and only if $i\leq j$ and
$a_{i}\leq a_{j}$. Given poset $Q$, its subset $B$ is a {\em chain} if the partial
order of $Q$ is a total order on $B$. Subset $C$ is an {\it antichain} if no two
elements $a,b$ of $C$ are comparable with respect to the partial order relation of
$Q$.  Dilworth's theorem (\cite{dilworth1950decomposition}) states that the minimum
number of classes in a chain decomposition of a partial order $Q$ is equal to the
size of the largest antichain of $Q$. 

One particular type of posets we consider is that of {\it permutation
  orders}: $\preceq$ is a \emph{permutation order} if there exists a
permutation $\pi$ of the elements of $U$ such that, for all $a,b\in U$, $a\preceq b$ iff $a<b$ and $\pi^{-1}(a)<\pi^{-1}(b)$. 

Another particular case we will be concerned with is that of {\em interval orders}. Without loss of generality all our intervals will be closed subsets of $(0,1)$. We define a partial order of them as follows: Given intervals $I_1=[a_1,b_1] $ and $ I_2=[a_2,b_2] $ with $a_1<b_1$ and $a_2<b_2$, we say that $I_1 \leq I_2$ if and only if the entire interval $I_1$ lies to the left of $I_2$ on the real numbers axis, that is $b_1 \leq a_2$. For technical reasons we will also require a total ordering of intervals, denoted by $\sqsubseteq$ and defined as follows: $I_{1}\sqsubseteq I_{2}$ if either $b_{1}<b_{2}$ or $b_{1}=b_{2}$ and $a_{1}<a_{2}$. 

We will consider in this paper models of random intervals. Similarly to the model in \cite{justicz1990random} by  "random intervals" we will mean  random subintervals of $(0, 1)$ generated as follows: A random sample $I$ can be constructed iteratively at each step by choosing two random real numbers $a,b\in (0, 1)$ and taking $I=[min(a,b),max(a,b)]$. 

The \emph{patience sorting algorithm}  (\cite{mallows1963patience}) partitions a permutation into a minimal number of increasing subsequences of integers. It works by adding each element in an online fashion to the first subsequence where it can be added, starting a new subsequence if no existing one is compatible. 

We need to briefly review one classical proof  of Dilworth's theorem, due to Fulkerson (\cite{fulkerson1956note}): First, given poset $Q=(U,\leq)$ with $n$ elements,  define the so-called {\it split graph associated to $Q$} (\cite{felsner2003recognition}), to be the bipartite graph $G_{Q}=(V_{1},V_{2},E)$, where $V_{1}=\{x^{-}:x\in U\}$ and $V_{2}=\{y^{+}:y\in U\}$ are independent copies of $U$, and given $x< y\in U$ we add to $E$ edge $x^{-}y^{+}$.
Fulkerson proved that each chain decomposition of $Q$ uniquely corresponds to a 
 matching in $G_{Q}$.  The proof proceeded by employing the classic K\H{o}nig-Egerv\'ary theorem (\cite{koniggrafok,egervary1931matrixok}), stating that the size of a maximum matching in a bipartite graph is equal to the minimum vertex cover in the same graph. This result was applied  to the bipartite graph $G_{Q}$, inferring that 

\begin{proposition} 
The cardinality of a minimum chain decomposition of $Q$ is equal to $n-vc(G_{Q})$, where $vc(G_{Q})$ is the vertex cover number of the split graph $G_{Q}$.   
\label{vc} 
\end{proposition} 

Finally, Fulkerson's  proof of Dilworth's theorem concluded by showing that $n-vc(G_{Q})$  is equal to the size of the largest antichain of $Q$.
 
We extend the definition of $k$-heapable sequences to general posets as follows: 
\begin{definition} 
Given integer $k\geq 1$ and poset $Q=(U,\leq)$ a subset $A$ of the ground set $U$ is called {\it $k$-heapable} (or, equivalently, a {\it $k$-ary chain} of $Q$) if there exists a $k$-ary rooted tree $T$ and a bijection between $A$ and the vertices of $T$ such that for every $i,j\in A$, if $j$ is a descendant of $i$ in $T$ then $i<j$ in $Q$. 
\end{definition} 

We stress the fact that the notion of a $k$-ary chain above is distinct from the notion of $k$-chain that appears in the statement of the Greene-Kleitman theorem (\cite{greene1976structure}). 

\begin{definition} 
The {\it $k$-width of poset $Q$}, denoted by $k$-wd(Q), was defined in \cite{istrate2016heapability} as the smallest number of classes in a partition of $Q$ into $k$-ary chains.
For $k=1$, by Dilworth's theorem, we recover the usual definition of poset width (\cite{trotter1995partially}). 
\end{definition} 

\begin{observation} In the previous definition we partition a \emph{set} into $k$-ary chains. On the other hand in the problem studied in \cite{istrate2015heapable} we partition a permutation (i.e. \emph{sequence} of elements) 
into $k$-heapable subsets. 

There is no contradiction between these two settings, as one can map a permutation $\pi\in S_{n}$ of $n$ elements onto a poset defined as $\{(i,\pi[i]):i=1,\ldots, n\}$ with 
\[
(i,\pi[i])< (j,\pi[j])\mbox{ iff }i<j\mbox{ and }\pi[i]<\pi[j]
\]
\end{observation}

Two important (unpublished) minimax theorems, attributed in \cite{gyarfas1985covering} to Tibor Gallai and ultimately subsumed by the statement that interval graphs are perfect, deal with sets of intervals. We will only be concerned with the first of them, that states that, given a set of intervals $J$ on the real numbers line the following equality holds:

\begin{proposition}
	The minimum number of partition classes of $J$ into pairwise disjoint intervals is equal to the maximum number of pairwise intersecting intervals in $J$. 
	\label{gallai1}
\end{proposition}

Of course, the first quantity in Proposition~\ref{gallai1} is nothing but the 1-width of the partial order $\leq$ on intervals. The scaling of (the expected value of) this parameter for sets $R$ of $n$ random intervals  has the form (\cite{justicz1990random})
\[
E[1\mbox{-wd(R)}] = \frac{2}{\sqrt{\pi}}\sqrt{n}(1+o(1)).
\]

\subsection{A graph-theoretic interpretation} 
\label{interpret} 

Problems that we are concerned with have, it turns out, an algorithmic interpretation that is strongly related to problems of computing the maximum independent sets and the chromatic number of various classes of perfect graphs. The idea first appeared (somewhat implicitly) in \cite{heapability-thesis}. In this subsection we make it fully explicit as follows: 
\begin{definition} 
Given a directed graph $G=(U,E)$, call a 
set of elements $W\subseteq U$ a \emph{$k$-treelike independent set of $G$} if there is a rooted $k$-ary tree $T$, not necessarily complete, and a bijection $f:V(T)\rightarrow W$ such that, for all vertices $v,w$ of $T$, if $v$ is an ancestor of $w$ in $T$ then $f(v)$ and $f(w)$ are {\bf not} connected by an edge in $G$. 

A poset can, of course, be viewed as a directed graph, so the previous definition applies to posets as well. 
We can apply the concept to \emph{undirected} graphs as well by identifying such a graph with its oriented version containing, for any undirected edge $e$, both directed versions of $e$. 
\end{definition} 

For undirected graphs a $1$-treelike independent set of $G$ is simply an independent set: indeed, the condition in the definition simply enforces the nonexistence of edges between the vertices in $W$. Paralleling the case $k=1$ (for which a polynomial time algorithm is known, in the form of patience sorting), we can restate the open problem from \cite{byers2011heapable} as the problem of computing a maximum $2$-treelike independent set: 

\begin{proposition}  Computing the  longest $k$-heapable subsequence of an arbitrary permutation (order) $\pi$ is equivalent to 
computing a maximum $k$-treelike independent sets in the digraph induced by the transitive closure of the Hasse diagram of the associated partial order $P_{\pi}$.  
\end{proposition}

When $k$ is implied, we will informally use the name \emph{maximum heapable subset} for the problem of computing a maximum $k$-ary chain of an arbitrary permutation. Similarly, the problem of partitioning the set into $k$-ary chains can be regarded as a generalization of graph coloring: define a  \emph{$k$-treelike coloring of a partial order $P$} to be a partition of the universe $U$ into classes that induce $k$-treelike independent sets. The {\it $k$-treelike chromatic number} of a partial order $P$ is the minimum number of colors in a $k$-treelike coloring of $P$.

The main result from \cite{istrate2015heapable} is equivalent to the following (re)statement:
\begin{proposition} 
For $k\geq 2$ the greedy algorithm of Figure~\ref{alg:greedy1} computes an optimal $k$-treelike coloring of permutation partial orders. 
\end{proposition}

\begin{figure}[h]
\begin{center}
	\fbox{
	    \parbox{12cm}{
		        \textbf{Input:} Permutation order $P$ specified by permutation $\pi\in S_n$.\\
				\textbf{Output:} A $k$-treelike coloring of $P$. \\
\\		
\textbf{let} $F=\emptyset$ \\
			\textbf{for} $i:=1$ to $n$ do:
			\begin{itemize}
				\item[] \textbf{if} $\pi(i)$ cannot become the child of any node in $F$. 				
				\begin{itemize}
					\item[] \textbf{then}  
					\item[] \hspace{1cm} start a new tree with root  $\pi(i)$.				
					\item[] \textbf{else} 
					\item[] \hspace{1cm} make $\pi(i)$ a child of the largest value $\pi(j)<\pi(i)$, 
					\item[] \hspace{1cm} $j<i$ that has fewer than $k$ children.  
				\end{itemize}
			\end{itemize}
	    }
	}
\end{center}
\caption{The greedy best-fit algorithm for $k$-tree-like colorings of permutation posets.} 
\label{alg:greedy1}
\end{figure} 

Again, let us remark that the corresponding statement for $k=1$ is well known, as it is equivalent to the greedy coloring of permutation graphs, accomplished by patience sorting.

\section{Main result} 
\label{sec:main}

Our main result proves a $k$-ary extension of Proposition~\ref{vc}. To state it, we extend the definition of split graphs to all values $k\geq 1$ as follows: 

\begin{definition} Given poset $Q=(U,\leq)$ an integer $k\geq 1$, the 
{\it $k$-split graph associated to poset $Q$} is  the bipartite graph $G_{Q,k}=(V_{1},V_{2},E)$, where 
\begin{itemize} 
\item $V_{1}=\{x^{-}_{1},\ldots, x^{-}_{k} :x\in U\}$ 
\item $V_{2}=\{y^{+}:y\in U\}$ 
\item given $x,y\in U,$ $x < y$, add to $E$ edge $x^{-}_{i}y^{+}$ for $i \in 1, ..., k$. 
\end{itemize} 
\end{definition} 

Given this definition, our main result shows that computing the $k$-width of a finite poset can be computed as in  Fulkerson's proof of Dilworth's theorem, by directly generalizing Proposition~\ref{vc}: 

\begin{theorem} 
Let $Q=(U,\leq)$ be a finite poset with $n$ elements and a fixed integer $k\geq 1$.  Then 
\begin{equation}
k\mbox{-wd(Q)} = n- \mbox{vc($G_{Q,k}$)}.
\end{equation}
\label{main:thm}
\end{theorem} 
\begin{proof}
Define a {\em left $k$-matching} of graph $G_{Q}$ to be any set of edges $A\subseteq E$ such that for every $x,y\in U$, $deg_{A}(x^{-})\leq k$ and $deg_{A}(y^{+})\leq 1$. 

\begin{claim}\label{partition_of_Q}
Partitions of $Q$ into $k$-ary chains bijectively correspond to left $k$-matchings of $G_{Q}$. The number of classes of a partition is equal to $n$ minus the number of edges in the associated left $k$-matching.  
\end{claim}
\begin{proof} 
Consider a left $k$-matching $A$ in $G_{Q}$. Define the partition $P_{A}$ as follows: roots of the $k$-ary chains consist of those $x\in U$ for which $deg_{A}(x^{+})=0$. There must be some element $x\in U$ satisfying this condition, as the minimal elements of $Q$ with respect to $\leq$ satisfy this condition. 

Now we recursively add elements of $U$ to the partition $P_{A}$ (in parallel) as follows: 
\begin{enumerate}
\item All elements $y\in U$ not yet added to any $k$-ary chain, and such that $y^+$ is connected to some $x^{-}$ by an edge in $A$ are added to the $k$-ary chain containing $x$, as direct descendants of $x$. Note that element $x$ (if there exists at least one $y$ with this property) is unique (since $deg_{A}(y^{+})=1$ in this case), so the specification of the $k$-ary chain to add $y$ to is well defined. On the other hand, each operation adds at most $k$ successors of any $x$ to its $k$-ary chain, since $deg_{A}(x^{-})\leq k$. 
\item If all direct predecessors of an element $x\in U$ have been added to some $k$-ary chain and are no longer leaves of that $k$-ary chain then $x$ will be the root of a new $k$-ary chain. 
\end{enumerate} 
Conversely, given any partition $P$ of $U$ into $k$-ary chains, define set of edges $A$ consisting of edges $x^{-},y^{+}$ such that $x$ is the parent of $y$ in a $k$-ary chain. It is immediate that $A$ is a left $k$-matching. 
\end{proof} 

\begin{corollary} \label{mapping_of_partition} There is a bijective mapping between 
partitions of $Q$ into $k$-ary chains and matchings of $G_{Q,k}$ such that the number of $k$-ary chains in a partition is $n$ minus the number of edges in the matching. 
\end{corollary}
\begin{proof}
We will actually show (using Claim~\ref{partition_of_Q}) how to associate left $k$-matchings of $G_{Q}$ to matchings of $G_{Q,k}$. The idea is simple: given a node $x^{-}$ of $G_{Q}$ with $l\leq k$ neighbors in $V_{2}$, construct a matching in $G_{Q,k}$ by giving each of $x^{-}_{1},\ldots, x^{-}_{l}$ exactly one neighbor from the neighbors of $x^{-}$ (in a pairwise distinct way). In the other direction, if $e=x_i^{-} y^{+}$ is an edge in the matching of $G_{Q,k}$ then consider the appropriate edge $x^{-}y^{+}$  in $G_Q.$ It is easy to check that it gives a left $k$-matching in $G_Q$ (which corresponds to a $k-$chain partition of  $Q).$ Furthermore, the number of $k-$chains is the number of edges in the left $k-$matching in $G_Q$ which is the same as $n$ minus $vc(G_Q).$ 
\end{proof}

We complete the proof of Theorem \ref{main:thm} (based on Claim \ref{partition_of_Q} and Corollary  \ref{mapping_of_partition})  by applying in a straightforward way  K\H{o}nig's theorem to the graph $G_{Q,k}$. 
\end{proof}

\begin{corollary} One can compute parameter $k$-wd(Q) by creating a flow network $Z_Q$ and computing the value of the maximum flow of $Z_Q$ consisting of :
\begin{itemize} 
\item vertices and edges of $G_{Q}$, with edge  capacity 1.
\item a source $s$, connected to nodes in $V_{1}$ by directed edges of capacity $k$, 
\item a sink $t$, that all nodes in $V_{2}$ connect to via oriented edges of capacity $1$. 
\end{itemize} 
computing the  maximum $s$-$t$ network flow value $f$ in network $Z_Q$ and outputting $k$-wd(Q)=n-f. 
\end{corollary}
\begin{proof}
Straightforward, this is simply the maximal flow algorithm for computing the maximal size left $k$-matching in $G_Q,$ similar to the construction for maximum matchings in bipartite graphs in the literature. 
\end{proof}
\section{Heapability of sequences of intervals: an online algorithm}
\label{sec:interval}

We now know that the problem of computing a minimal partition of the elements of a poset into heapable subsequences has a polynomial time algorithm. 

On the other hand for permutations (sequences of integers) the optimal algorithm presented in \cite{istrate2015heapable}, which extended the well-known
{\it patience sorting algorithm}  (\cite{mallows1963patience}), had a simple, online, structure: the heaps were built incrementally, by considering the elements one by one and inserting each element as a leaf in one of the existing heaps, or starting a new heap. In contrast, the particularization of the network flow algorithm to permutations is  {\bf not}  online in any natural sense, as it computes a maximum flow in a graph which depends "globally" on the sequence. One could ask whether the existence of the online algorithm from \cite{istrate2015heapable} is an exception or, rather, such online algorithms which are optimal exist for sequences of elements from other posets as well. 

The question is also interesting since the case $k=1$ is a natural variant of the \emph{online chain partitioning problem}, a problem with a rich history (\cite{kierstead1984theory,trotter1995partially}) and bibliography (see e.g. \cite{bosek2012line}).  In contrast to the classical case, in our variant elements can be assigned to a ($k$-ary) chain only if they can be inserted \emph{as leaves} at the moment they are inserted.  

In the sequel we provide an affirmative answer to the above question, by focusing on the case of sequences of intervals:  

\begin{theorem}
	For every fixed $k \geq 1$ there exists a (polynomial time) online algorithm that, given a sequence of intervals $S = ( I_1, I_2, ..., I_n )$ as input, computes a minimal partition of $S$ into $k$-ary chains so that each interval is inserted at the time of its consideration as a leaf into one of the $k$-ary chains, or starts a new $k$-ary chain. 
	\label{intervals}
\end{theorem}

Before proceeding with the proof of theorem~\ref{intervals}, let us remark a potential application of a variant of this result to parallel computing: many algorithms in this area (e.g. algorithms using a {\em parallel prefix-sum} design methodology, \cite{BlellochTR90}) require the computation of all prefixes of an associative operation $A_{1}*A_{2}*\ldots A_{n}$. Operations 
 being performed (each corresponding to one computation of a *-product) are arranged on a binary tree. In the (completely equivalent) max-heap variant of theorem~\ref{intervals}, children intervals are required to be less or equal to the parent interval with respect to ordering $\leq$. This is quite natural from the standpoint of parallel computing: consider the setting of a parallel-prefix problem where each intermediate *-computation is a rather costly operation; intervals now represent times when these operations can be scheduled. The requirement that the parent interval be larger than child intervals with respect to $\leq$ is completely natural, as child computations need to complete in order to feed their results to the parent computation. Thus our result answers the question whether all the time intervals can be scheduled on a single heap-ordered binary tree, and gives such a scheduling, if the answer is affirmative. 
\subsection{Proof of Theorem \ref{intervals}}
\label{sec:heap-proof}

The proof  employs the concept of {\em slots}, adapted for interval sequences  from similar concepts for permutations (\cite{byers2011heapable,istrate2015heapable}): 

\begin{definition}
When a new interval is added to a $k$-ary chain it opens k new positions to possibly insert other intervals as direct successors into this node. Each position has an associated integer value that will be called its {\em slot}. The value of all empty slots created by inserting $I_1=[a_1,b_1]$ into a $k$-ary chain will be $b_1$, the right endpoint of $I_1$. 
\end{definition}

\begin{definition}
An interval $I$ is {\it compatible with an (empty) slot with value $x$} if all of $I$ lies in $[x,\infty).$
\end{definition}

Intuitively, $x$ is the smallest value of the left endpoint of an interval that can be inserted in the $k$-ary chain as a child of an interval $I_{1}$ with right endpoint $x$ while respecting the heap property.  Indeed, as $k$-ary chains are (min-)heap ordered, an insertion of an interval $I$ into a $k$-ary chain as a child of $I_1$ is legal if the interval $I$ is greater than $I_{1}$ with respect to $\leq$ relation. This readily translates to the stated condition, that the start point  of $I$ must be greater or equal than the slot value of its parent.

\begin{figure}[ht]
\begin{center}
	\fbox{
	    \parbox{11cm}{
		        \textbf{Input:} A sequence of intervals $I=(I_{1},I_{2},\ldots, I_{n})$.\\
				\textbf{Output:} A partition $H$ of $I$ into $k$-ary chains. \\
		
			\textbf{for} $i:=1$ to $n$ do:
			\begin{itemize}
				\item[] \textbf{if} $I_{i}=[l_{i},r_{i}] $ can be inserted into some empty slot   							
				\begin{itemize}
					\item[] \textbf{then} insert $I_i$ in the highest-valued compatible slot (a child of the node with this slot).				
					\item[] \textbf{else} create a new $k$-chain rooted at $I_i$ 
				\end{itemize}
			\end{itemize}
	    }
	}
\end{center}
\caption{The best-fit algorithm for $k$-ary chain partition of sequences of intervals.} 
\label{alg:greedy}
\end{figure} 

\begin{figure}[!ht]
\resizebox{1.\linewidth}{!}{
	\begin{tabular}{cc}
		\begin{tikzpicture}[->,>=stealth',shorten >=1pt,auto,node distance=0.55cm,
		thick,main node/.style={rectangle,draw,font=\sffamily\small\bfseries}]
		
		\node[label=90:$H_1$][main node] (1) {$[1, 7]$};
		\node[main node] (2) [below left = of 1] {$[7, 9]$};
		\node[main node] (3) [below right = of 1] {$[8, 16]$};
		\node[main node] (4) [below left = of 2] {};
		\node[main node] (5) [below right = of 2] {};
		\node[main node] (6) [below left = of 3] {};
		\node[main node] (7) [below right = of 3] {};
		
		\path[every node/.style={font=\sffamily\small}]
		(1) edge node [left, pos=0.2] {7} (2)
		edge node[right, pos=0.2] {7} (3)
		(2) edge node [left, pos=0.2] {9} (4)
		edge node [right, pos=0.2] {9} (5)
		(3) edge node [left, pos=0.2] {16} (6)
		edge node [right, pos=0.2] {16} (7);
		
		\end{tikzpicture} &
		\hspace{-1cm}\begin{tikzpicture}[->,>=stealth',shorten >=1pt,auto,node distance=0.75cm,
		thick,main node/.style={rectangle,draw,font=\sffamily\small\bfseries}]
		
		\node[label=90:$H_3$] [main node] (1) {$[1, 2]$};
		\node[main node] (2) [below left = of 1] {$[3, 19]$};
		\node[main node] (3) [below right = of 1] {$[5, 7]$};
		\node[main node] (4) [below left = of 2] {};
		\node[main node] (5) [below right = of 2] {};
		\node[main node] (6) [below left = of 3] {};
		\node[main node] (7) [below right = of 3] {};
		
		\path[every node/.style={font=\sffamily\small}]
		(1) edge node [left, pos=0.2] {2} (2)
		edge node[right, pos=0.2] {2} (3)
		(2) edge node [left, pos=0.2] {19} (4)
		edge node [right, pos=0.2] {19} (5)
		(3) edge node [left, pos=0.2] {7} (6)
		edge node [right, pos=0.2] {7} (7);
		
		\end{tikzpicture} \\

	\hspace{2cm}
	\begin{tikzpicture}[->,>=stealth',shorten >=1pt,auto,node distance=0.75cm,
	thick,main node/.style={rectangle,draw,font=\sffamily\small\bfseries}]
	
	\node[label=90:$H_2$] [main node] (1) {$[1, 11]$};
	\node[main node] (2) [below left = of 1] {$[11, 12]$};
	\node[main node] (3) [below right = of 1] {};
	\node[main node] (4) [below left = of 2] {$[15, 16]$};
	\node[main node] (5) [below right = of 2] {$[13, 17]$};
	\node[main node] (6) [below left = of 4] {};
	\node[main node] (7) [below right = of 4] {};
	\node[main node] (8) [below left = of 5] {};
	\node[main node] (9) [below right = of 5] {};
	
	\path[every node/.style={font=\sffamily\small}]
	(1) edge node [left, pos=0.2] {11} (2)
	edge node[right, pos=0.2] {11} (3)
	(2) edge node [left, pos=0.2] {12} (4)
	edge node [right, pos=0.2] {12} (5)
	(5) edge node [left, pos=0.2]{17} (8)
	edge node [right, pos=0.2]{17} (9)
	(4) edge node [left, pos=0.2] {16} (6)
	edge node [right, pos=0.2] {16} (7);
	
	\end{tikzpicture}
	& \\
	\end{tabular}
}
	\caption{The binary (2-ary)-chain configuration corresponding to $S_1$.}
	\label{fig:17}
\end{figure}

The proposed greedy best-fit algorithm for computing a  minimum partition into $k$-chains a sequence of $n$ intervals  is described in Figure~\ref{alg:greedy}. As an example, consider the sequence of intervals $S_1$ below, with $k=2$. The resulting  configuration is shown in Figure \ref{fig:17}.

\[S_1=\left([1, 7], [1, 11], [11, 12], [15, 16], [7, 9], [8, 16], [1, 2], [3, 19], [13, 17], [5, 7]\right)\]

By choosing the highest valued slot available for insertion for $I_t=[a_t,b_t]$, we make sure that the difference between the chosen slot value $s$ and $a_t$ is minimal. This is desirable because there may be some interval further down the sequence with starting  point value in between the $s$ and $a_t$ that fits slot $s$ but cannot be inserted there, as the slot is no longer available. 

We define the concepts of signature of a multiset of slots and domination between such multisets in a similar way to the corresponding concepts for permutations in \cite{byers2011heapable}: 
\begin{definition}
	Given a multiset of slots $H,$ we call the vector of slots, sorted in increasing order, the \emph{signature} of $H$. We shall denote this by $sig(H)$.	 By slightly abusing notation, we will employ the previous definition in the obvious way when $H$ is a union of $k$-chains as well. 
\end{definition}
\begin{definition}
	Multiset  $P$ {\it dominates multiset $Q$} (denoted $P \preceq Q$) if \\ $|sig(P)|\leq |sig(Q)|$ and for all $1\leq i \leq |sig(P)|$ we have $sig(P)[i] \leq sig(Q)[i]$. 
\end{definition}
 
For example, for the binary ($2$-ary) chains in the Figure \ref{fig:17} their corresponding signatures are, respectively:

	\begin{itemize}
		\item[] $sig(H_1) = [9, 9, 16, 16]$;
		\item[] $sig(H_2) = [11, 16, 16, 17, 17]$;
		\item[] $sig(H_3) = [7, 7, 19, 19]$.
	\end{itemize}

Therefore, in our example $H_1 \preceq H_2$, ($H_1$ dominates $H_2$), but no other domination relations between $H_{1},H_2,H_3$ are true. 

\begin{lemma}
	Assume that $A,B$ are multisets of slots and $A$ dominates $B$. Then the multisets $A'$ and $B'$ obtained after inserting a new interval  into the largest compatible slot of $A$, and into an arbitrary compatible slot of $B$, propagates the domination property, i.e. $A'$ dominates $B'$.
	\label{lemma:1}
\end{lemma}

\begin{proof} 

Let $sig(A)=[a_{1},a_{2},\ldots, a_{|sig(A)|}]$ and $sig(B)=[b_{1},b_{2},\ldots, b_{|sig(B)|}].$ Also, by convention, define $a_{0}=b_{0}=-\infty$ and $a_{|sig(A)|+1}=b_{|sig(B)|+1}=+\infty$. 

Proving that $A' \preceq B'$ is equivalent to proving that $|sig(A')|\leq |sig(B')|$ and for all indices $1 \leq l \leq |sig(A')|$:
\begin{equation}
	sig(A')[l] \leq sig(B')[l].
\end{equation}
The cardinality condition can be easily verified: indeed, $|sig(A')|$ is either $|sig(A)|+k-1$ (if the process adds $k$ copies of $y$ but also kills a lifeline) or $|sig(A)|+k$ (no slot exists with a value less or equal to $x$), and similarly for $|sig(B')|.$    It follows easily from the domination property that when $|sig(A')|=|sig(A)|+k$ then $|sig(B')|=|sig(B)|+k$ as well. Indeed, since $|sig(A')|=|sig(A)|+k$, no slot of $A$ has value lower than $x$ (or it would lose a lifeline). Thus $a_{1}>x$ and since $A\preceq B$, 
$b_{1}\geq a_{1}>x$. So no slot of $B$ is smaller than $x$ either. 

As for the second condition, consider the slots from $A$ and $B$ which interfere with the newly arrived interval as follows: 

\begin{itemize}
	\item Let  $i$ and $i'$ be the (unique) indices such that 
	$a_i \leq x < a_{i+1}$ and $a_{i'} \leq y < a_{i'+1} $ hold in $A,$ with $i,i' \in \{0,\cdots, |sig(A)|\}.$
	\item Similarly, let $j$ and $j'$ be the unique indices such that 
	$b_j \leq x < b_{j+1}$ and $b_{j'} \leq y < b_{j'+1}$ hold in $B,$ with $j,j' \in \{0,\cdots, |sig(B)|\}.$
\end{itemize}

\begin{figure} 
\resizebox{.99\linewidth}{!}{
\begin{picture}(430, 170)
	\linethickness{0.523cm}
	\color{pastelred}
	\put(279, 151){\line(0,-1){12}}

	\linethickness{0.51cm}
	\color{pastelred}
	\put(78, 116){\line(0,-1){12}}
	
	\linethickness{1.03cm}
	\color{pistachio}
	\put(342, 71){\line(0,-1){12}}
	
	\linethickness{1.03cm}
	\color{pistachio}
	\put(187, 36){\line(0,-1){12}}

	\color{black}
	\linethickness{0.02cm}

	
	\put(-35, 140) {$sig(A)$}
	
	\put(0, 151) {\line(1,0){400}}
	\put(0, 139) {\line(1,0){400}}
	\put(0, 151) {\line(0,-1){12}}
	\put(400, 151) {\line(0,-1){12}}
	
	\put(264, 142) {$a_i$}
	\put(260, 151) {\line(0,-1){12}}
	\put(275, 151) {\line(0,-1){12}}
	
	\put(326, 142) {$a_{i'}$}
	\put(323, 151) {\line(0,-1){12}}
	\put(337, 151) {\line(0,-1){12}}

	
	\put(-35, 105){$sig(B)$}
	
	\put(0, 116) {\line(1,0){400}}
	\put(0, 104) {\line(1,0){400}}
	\put(0, 116) {\line(0,-1){12}}
	\put(400, 116) {\line(0,-1){12}}

	\put(66, 106) {$b_t$}
	\put(63, 116) {\line(0,-1){12}}
	\put(77, 116) {\line(0,-1){12}}

	\put(124, 106) {$b_j$}
	\put(120, 116) {\line(0,-1){12}}
	\put(135, 116) {\line(0,-1){12}}
	
	\put(176, 106) {$b_{j'}$}
	\put(173, 116) {\line(0,-1){12}}
	\put(187, 116) {\line(0,-1){12}}

	\put(68, 165) {$t$}
	
	\put(70, 155) {\line(0,-1){10}}
	\put(70, 140) {\line(0,-1){10}}
	\put(70, 125) {\line(0,-1){10}}
	\put(70, 110) {\line(0,-1){10}}
	\put(70, 95) {\line(0,-1){10}}
	\put(70, 80) {\line(0,-1){10}}
	\put(70, 65) {\line(0,-1){10}}
	\put(70, 50) {\line(0,-1){10}}
	\put(70, 20) {\line(0,-1){10}}

	\put(178, 165) {$j'$}
	
	\put(180, 155) {\line(0,-1){10}}
	\put(180, 140) {\line(0,-1){10}}
	\put(180, 125) {\line(0,-1){10}}
	\put(180, 110) {\line(0,-1){10}}
	\put(180, 95) {\line(0,-1){10}}
	\put(180, 80) {\line(0,-1){10}}
	\put(180, 65) {\line(0,-1){10}}
	\put(180, 50) {\line(0,-1){10}}
	\put(180, 20) {\line(0,-1){10}}

	\put(328, 165) {$i'$}

	\put(330, 155) {\line(0,-1){10}}
	\put(330, 140) {\line(0,-1){10}}
	\put(330, 125) {\line(0,-1){10}}
	\put(330, 110) {\line(0,-1){10}}
	\put(330, 95) {\line(0,-1){10}}
	\put(330, 80) {\line(0,-1){10}}	
	\put(330, 65) {\line(0,-1){10}}
	\put(330, 50) {\line(0,-1){10}}
	\put(330, 35) {\line(0,-1){10}}
	\put(330, 20) {\line(0,-1){10}}

	
	\put(-35, 60) {$sig(A')$}
	
	\put(0, 71) {\line(1,0){400}}
	\put(0, 59) {\line(1,0){400}}
	\put(0, 71) {\line(0,-1){12}}
	\put(400, 71) {\line(0,-1){12}}

	\put(124, 62) {$a_j$}
	\put(120, 71) {\line(0,-1){12}}
	\put(135, 71) {\line(0,-1){12}}

	\put(264, 62) {$a_{i+1}$}
	\put(260, 71) {\line(0,-1){12}}
	\put(275, 71) {\line(0,-1){12}}

	\put(312, 62) {$a_{i'}$}
	\put(308, 71) {\line(0,-1){12}}

	\put(328, 62) {$y$}
	\put(323, 71) {\line(0,-1){12}}
	\put(337, 71) {\line(0,-1){12}}
	
	\put(343, 62) {$y$}
	\put(352, 71) {\line(0,-1){12}}
	
	\put(356, 62) {$a_{i'+1}$}
	\put(367, 71) {\line(0,-1){12}}

	
	\put(-35, 25){$sig(B')$}
	
	\put(0, 36) {\line(1,0){400}}
	\put(0, 24) {\line(1,0){400}}
	\put(0, 36) {\line(0,-1){12}}
	\put(400, 36) {\line(0,-1){12}}

	\put(66, 26) {$b_{t+1}$}
	\put(63, 36) {\line(0,-1){12}}
	\put(77, 36) {\line(0,-1){12}}

	\put(109, 26) {$b_j$}
	\put(105, 36) {\line(0,-1){12}}

	\put(124, 26) {$b_{j+1}$}
	\put(120, 36) {\line(0,-1){12}}
	\put(135, 36) {\line(0,-1){12}}

	\put(161, 26) {$b_{j'}$}
	\put(158, 36) {\line(0,-1){12}}

	\put(178, 27) {$y$}
	\put(173, 36) {\line(0,-1){12}}
	\put(187, 36) {\line(0,-1){12}}		
	
	\put(193, 27) {$y$}
	\put(202, 36) {\line(0,-1){12}}
	
	\put(206, 26) {$b_{j'+1}$}
	\put(217, 36) {\line(0,-1){12}}
	
	\put(342, 26) {$b_{i'}$}
	\put(337, 36) {\line(0,-1){12}}
	\put(352, 36) {\line(0,-1){12}}
	
	\put(357, 26) {$b_{i'+1}$}
	\put(367, 36) {\line(0,-1){12}}

	
	\put(30, 10) {\ding{192}}
	\put(120, 10) {\ding{193}}
	\put(255, 10) {\ding{194}}
	\put(365, 10) {\ding{195}}

\end{picture}
}
\caption{The various cases of inserting a new interval [x, y] into $B$ and $A$.} 
\label{fig:cases} 
\end{figure} 

Since $A \preceq B$, it follows that $i \geq j$ and $ i' \geq j'$.
Suppose that we insert the interval $[x, y]$ in $B$ in an arbitrary slot $b_t \leq b_j$ (thus removing one life of slot $b_t$ and inserting $k$ copies of slot $y$). The rest of the proof is by a case analysis. The four cases which can be distinguished (displayed in Fig.~\ref{fig:cases}, for $k=2$) are:
\begin{itemize}
	\item[Case 1.] $l < t$: \\
		In this case, none of the  signatures $A,B$ were affected at position $l$ by the insertion of $[x, y]$, hence:
		\begin{center}
			$ sig(A')[l] = sig(A)[l] \leq sig(B)[l] = sig(B')[l] $.
		\end{center}	
		
	\item[Case 2.] $l \in [t, j')$: \\
		In this range all slots from $B'$ have been shifted by one position to the left compared to $B$ due to the removal of $b_t$. Consequently:
		\begin{center}
			$ sig(A')[l] = sig(A)[l] \leq sig(B)[l+1] = sig(B')[l] $.
		\end{center}
		
	\item[Case 3.] $l \in [j', i')$: \\
		Knowing that $i'$ is the position in $A$ where we inserted $k$ new slots with value $y$, $j'$ is the position in $B$ where we inserted these same slots and $i' \geq j'$, the following is true:
		\begin{center}
			$ sig(A')[l] \leq sig(B)[j'] \leq y = sig(B')[i'] \leq sig(B')[l]$.
		\end{center}
	
	\item[Case 4.] $l \geq i'$: 
		\begin{itemize}
			\item[$a)$] For $ l = i', \ldots,  i'+k-1 $ : $ sig(A')[l] = sig(B')[j'] = y $. Since $ j' < i' < i'+1 $, then:
			\begin{center}
				$ sig(A')[i'] = y = sig(B')[j'] \leq sig(B')[i'] $.\\
				$ sig(A')[i' + 1] = y = sig(B')[j'] \leq sig(B')[i' + 1] $.
			\end{center} 
			
			\item[$b)$]	For $ l \geq i' + k $ the two signatures have equally shifted components compared to the original signatures of $A$ and $B$, so:
			\begin{center}
				$ sig(A')[l] = sig(A')[i' +k-1+ (l-i-k+1)]=sig(A)[i' +(l-i'-k+1)]  \leq sig(B)[l-k+1] = sig(B')[l]$
			\end{center}
			
		\end{itemize}
\end{itemize}

In conclusion, $ sig(A')[l] \leq sig(B')[l]$ for any $l$, and relation  $ A' \preceq B' $ follows. 
\end{proof}

\begin{lemma}
Given a sequence $I$ of intervals, consider an optimal way of partitioning $I$ into $k$-ary chains. Let $r$ be a stage where the greedy best-fit algorithm creates a new $k$-ary chain. Then the optimal way also creates a new $k$-ary chain. 
\label{bar}
\end{lemma} 
\begin{proof}
We use the fact that, by Lemma \ref{lemma:1}, before every step $r$ of the algorithm the multiset $\Gamma_{r-1}$ of slots created by our greedy algorithm dominates the multiset $\Omega_{r-1}$ created by the optimal insertion. 

Suppose that at stage $r$ the newly inserted interval $I_{r}$  causes a new $k$-ary chain to be created. That means that the left endpoint $l_{r}$ of $I_{r}$ is lower than any of the elements of the multiset $\Gamma_{r-1}$. By domination, the minimum slot of $\Omega_{r-1}$ is at least as high as the minimum slot of $\Gamma_{r-1}$. Therefore $l_{r}$ is also lower than any slot of $\Omega_{r-1}$, which means that the optimal algorithm also creates a new $k$-ary chain when inserting $I_{r}$. 
\end{proof}

Lemma \ref{bar} proves that the best-fit algorithm for insertion of new intervals is optimal. 

\subsection{The interval Hammersley (tree) process} 

One of the most fruitful avenues for the investigation of the scaling properties of the LIS (Longest increasing subsequence) of a random permutation is made via the study of the  (so-called {\it hydrodynamic}) limit behavior of an interacting particle system known as {\it Hammersley's process} ( \cite{aldous1995hammersley}). This is a stochastic process that, for the purposes of this paper can be defined (in a somewhat simplified form) as follows: random numbers $X_{0},X_{1}, \ldots, X_n,\ldots \in (0, 1)$ arrive at integer moments. Each value $X_{j}$ eliminates ("kills") the smallest $X_i > X_j$ that is still alive at moment $j$. Intuitively, "live" particles represent the top of the stacks in the  {\it patience sorting algorithm} that computes parameter LIS. 

The problem of partitioning a random permutation into a minimal set of $k$-heapable subsequences is similarly connected to a variant of the above process, introduced in \cite{istrate2015heapable} and further studied in \cite{basdevant2016hammersley,basdevant2017almost}, where it was baptized {\it Hammersley's tree process}. Now particles come with $k$ {\it lives}, and each particle $X_{j}$ merely takes one life of the smallest $X_{i}>X_{j}$, if any (instead of outright killing it). 

The proof of theorem \ref{intervals} shows that a similar connection holds for sequences of {\it intervals}. The {\em Hammersley interval tree process} is defined as follows: "particles" are still numbers in $(0, 1)$, that may have up to $k$ lives. However now the sequence $I_{0},I_{1},\ldots, I_{n},\ldots $ is comprised of random {\it intervals} in $(0, 1)$. When interval $I_{n}=[a_{n},b_{n}]$ arrives, it is $a_{n}$ that takes a life from the largest live particle $y \leq  a_{n}$. However, it is $b_{n}$ that is inserted as a new particle, initially with $k$ lives. 

\begin{corollary} 
Live particles in the above Hammersley interval process correspond to slots in our greedy insertion algorithm above. The newly created $k$-ary chains correspond to local minima (particle insertions that have a value lower than the value of any particle that is alive at that particular moment). 
\label{cor:1}
\end{corollary}

\section{From sequences to sets of intervals}
\label{sec:sets}
Theorem~\ref{intervals} dealt with {\em sequences} of intervals. On the other hand a set of intervals does not come with any particular listing order on the constituent intervals. Nevertheless, the problem can be easily reduced to the sequence case by the following: 

\begin{theorem} 
Let $k\geq 1$ and $Q$ be a set of intervals. Then the $k$-width of $Q$ is equal to the $k$-width of $Gr_{Q}$, the sequence of intervals obtained by listing the intervals in the increasing order of their {\it right endpoints} (with earlier starting intervals being preferred in the case of ties). 
\label{sets}
\end{theorem}  

\begin{proof}
Clearly $k$-wd(Q)$\leq k$-wd($Gr_{Q}$), since a partition of $Gr_{Q}$ into $k$-ary chains is also a partition of $Q$.  

To prove the opposite direction we need the following: 

\begin{lemma}
Let $S$ be a multiset of slots. Let $I_{1}=[a_{1},b_{1}],  I_{2}=[a_{2},b_{2}]$ be two intervals such that $b_{1}<b_{2},$ or $b_{1}=b_{2}$ and $a_{1}<a_{2}$. Let $S_{1}$ and $S_{2}$ be the multisets of slots obtained by inserting the two intervals in the order $(I_{1},I_{2})$ and ($I_{2},I_{1}$), respectively. Then $S_{1}$  dominates $S_{2}$. 
\label{domin:order}
\end{lemma}
\begin{proof}

The only nontrivial cases are those for which $S_{1}\neq S_{2}$. This condition can only happen when the insertions of $I_{1},I_{2}$ ``interact''. 

Indeed, let $x,y$ be the values of the largest slots less or equal to $a_{1},a_{2}$, respectively. If $x\neq y$ and $b_{1}$ does not become (after insertion of $I_{1}$) the slot occupied by  $b_{2}$ instead of $y$ (nor does the symmetric situation for the insertion order $I_{2},I_{1}$ hold) then $S_{1},S_{2}$ are obtained by adding $k$ copies of $b_{1},b_{2}$ each, and deleting $x,y$, so we obtain $S_{1}=S_{2}$. 

If instead $x=y$ then the two slots to be deleted (for both insertion orders) are $x$ and the largest slot smaller or equal than $x$. We obtain again $S_{1}=S_{2}$.

The only remaining case is when the slot removed after insertion of $I_{2}$ is $b_{1}$. In this case $S$ cannot contain any element in the range $(b_{1},a_{2})$.  Insertion $(I_{1},I_{2})$ removes $x$, adds $k-1$ copies of $b_{1}$ and $k$ copies of $b_{2}$.  On the other hand, insertion $(I_{2},I_{1})$ may remove some element $x^{\prime}$. It is certainly $x\leq x^{\prime}\leq b_{1}$. It then adds $k$ copies of $b_{2}$.  Then it removes some $x^{\prime \prime}\leq x$ and adds $k$ copies of $b_{1}$. Thus both sets $S_{1},S_{2}$ can be described as adding $k$ copies of $b_{1}$ and $k$ copies of $b_{2}$ to $S$, and then deleting one or two elements, two of them
\begin{itemize}
\item[-] $x$ and $b_{1}$ in the case of order ($I_{1},I_{2}$)
\item[-] $x^{\prime}$ and $x^{\prime \prime}$ in the case of order ($I_{2},I_{1}$).
\end{itemize} 
in the case when some element of $S$ is strictly less than $a_{1}$, and one element
\begin{itemize}
\item[-] $b_{1}$ in the case of order ($I_{1},I_{2}$)
\item[-] $x^{\prime \prime}$ in the case of order ($I_{2},I_{1}$)
\end{itemize} 
otherwise. In the second case the result follows by taking into account the fact that $x^{\prime \prime}\leq b_{1}$ and the following lemma:
\begin{lemma}
Let $S$ be a multiset of slots. Let $s_{2}, s_{1}\in S$ with $s_{2}\leq s_{1}$, and let $S_{1},S_{2}$ be the sets obtained by deleting from $S$ elements $s_{1}$ (or $s_{2}$, respectively). Then $S_{1}$ dominates $S_{2}$. 
\label{domin}
\end{lemma} 
\begin{proof} It follows easily from the simple fact that for every $r\geq 1$ and multiset $W$ the function that maps an integer $x\in W$ to the $r$'th smallest element of $W\setminus x$ is non-increasing (more precisely a function that jumps from the $r+1$'th down to the $r$'th smallest element of $W$)
\end{proof}
The first case is only slightly more involved. Since $x^{\prime}\leq x$, by applying Lemma \ref{domin} twice we infer: 
\begin{align*}
S_{1} & = & S\setminus \{x,b_{1}\}=(S\setminus \{b_{1}\})\setminus \{x\}\preceq (S\setminus \{b_{1}\})\setminus \{x^\prime\}=
(S\setminus \{x^\prime\})\setminus \{b_{1}\}\preceq \\  & \preceq & (S\setminus \{x^\prime\})\setminus \{x^{\prime\prime}\} =  S_{2}.\\
\end{align*}
which is what we wanted to prove. 
\end{proof}

From Lemma \ref{domin:order} we infer the following result
\begin{lemma} Let $1\leq r\leq n$ and let $X,Y$ be two permutations of intervals $I_{1}$, $I_{2},\ldots, I_{n}$, 
\begin{align*}
X=(I_{1},\ldots, I_{r-1},I_{r},I_{r+1}, \ldots I_{n}),\\
Y=(I_{1},\ldots, I_{r-1},I_{r+1},I_{r}, \ldots I_{n}).\\
\end{align*}
respectively (i.e. $X,Y$ differ by a transposition). If $I_{r}\leq I_{r+1}$ (recall, this means that the right endpoint of $I_{r}$ is less or equal than the left endpoint of $I_{r+1}$) then multisets of slots $S_{X},S_{Y}$ obtained by inserting intervals according to the listing specified by $X$ and $Y,\mbox{ respectively, }$ satisfy
\[
S_{X}\preceq S_{Y}.
\]
\label{domin3}
\end{lemma} 
\begin{proof} 
Without loss of generality one may assume that $r=n-1$ (as the result for a general $n$ follows from this special case by repeatedly applying Lemma \ref{lemma:1}). Let $S$ be the multiset of slots obtained by inserting (in this order) intervals $I_{1},\ldots, I_{r-1}$. Applying Lemma \ref{domin:order} to intervals $I_{r},I_{r+1}$ we complete the proof of Lemma \ref{domin3}. 
\end{proof}

Now the opposite direction in the proof of Theorem \ref{sets} follows: the multiset $S_{Gr(Q)}$ of labels obtained by inserting the intervals according to sequence $Gr_{Q}$ dominates any multiset of labels arising from a different permutation, since one can "bubble down" smaller intervals (as in bubble sort), until we obtain $Gr_{Q}$. As we do so, at each step, the new multiset of labels dominates the old one. Hence $S_{Gr(Q)}$ dominates all multisets arising from permutations of $I_{1},\ldots, I_{n}$, so sequence $Gr_{Q}$ minimizes the parameter $k$-wd among all permutations of $Q$.  
\end{proof} 
\begin{corollary} 
The greedy algorithm in Figure \ref{alg:greedy2} computes the $k$-width of an arbitrary set of intervals. 
\end{corollary} 

\begin{figure}
\begin{center}
	\fbox{
	    \parbox{11cm}{
		        \textbf{Input:} A set of intervals $I$.\\
				\textbf{Output:} A partition $H$ of $I$ into $k$-ary chains. \\
{\bf Sort} the intervals w.r.t. $\sqsubseteq$: $I=(I_{1},\ldots, I_{n})$. \\		
			\textbf{For} $i:=1$ to $n$ do:
			\begin{itemize}
				\item[] \textbf{If} $I_{i}=[a_{i},b_{i}] $ can be inserted into some empty slot   							
				\begin{itemize}
					\item[] \textbf{then} insert $I_i$ in the highest-valued compatible slot.				
					\item[] \textbf{else} create a new $k$-chain rooted at $I_i$ 
				\end{itemize}
			\end{itemize}
	    }
	}
\end{center}
\caption{The greedy algorithm for sets of intervals.} 
\label{alg:greedy2}
\end{figure} 

\begin{corollary} Modify the Hammersley interval process to work on sets of intervals by considering them in non-decreasing order according to relation $\sqsubseteq$. Then live particles in the modified process correspond to slots obtained using our greedy insertion algorithm for sets of intervals. New $k$-ary chains correspond to local minima (such particle insertions that have a value lower than the value of any particle that is alive at that particular moment). 
\label{cor:2}
\end{corollary} 

\section{Extension to sequences of elements from a trapezoid partial order} 
\label{trapez} 

The theorem in the previous section is strongly reminiscent of the fact that for interval partial orders a greedy best-fit algorithm computes the 
chromatic number (\cite{olariu1991optimal}). This result has an extension to an even more general class of graphs, that of \emph{trapezoid graphs} (\cite{dagan1988trapezoid}).  Trapezoid graphs are an extension of both interval and permutation graphs that unify many natural algorithms (for problems such as maximum independent set, coloring) for the two classes of graphs. 

As noted in \cite{felsner1997trapezoid}, trapezoid graphs can be equivalently defined as the cocomparability graphs of two-dimensional boxes, with sides parallel to the coordinate axes: 

\begin{definition} A \emph{box} is the set of points in ${\bf R}^{2}$ defined as 
\[
B=\{(x_{1},x_{2})\in \textbf{R}^2\mbox{ }|\mbox{ } l_{i}^{B}\leq x_{i}\leq r_{i}^{B},i=1,2\}
\]
for some numbers $l_{i}^{B}\leq r_{i}^{B}\in \textbf{R}$, where $l_{B}=(l_{1}^{B},l_{2}^{B})$ is \emph{the lower corner} of the box $B$ and $r_{B}=(r_{1}^{B},r_{2}^{B})$ is \emph{the upper corner}. 

The \emph{interval   $I(B)$ associated to box $B$} is the projection onto the $y$ axis of box $B$. Clearly $I(B)=[l^{B}_{2},r^{B}_{2}]$. 
\end{definition} 

The dominance partial order on intervals naturally extends to boxes: 

\begin{definition} The \emph{dominance partial order among boxes} is defined as follows: box \textrm{$B_{1}$ dominates box $B_{2}$} if 
point $r_{B_{1}}$ dominates point  $l_{B_{2}}$, i.e. 
\[
r_{i}^{B_{1}}\leq l_{i}^{B_{2}}\mbox{ for } i=1,2. 
\]
A \textrm{trapezoid partial order} is a poset $P$ induced by the dominance partial order on a finite set $V$ of boxes. $l$ and $u$ refer to the lower and upper corner coordinate functions defining the boxes. That is, for every $v\in V$, $l(v)$ is the lower corner of box $v$ and $u(v)$ is its upper corner. 
\end{definition} 

The box representation allowed the authors of \cite{felsner1997trapezoid} to give a ``sweep-line algorithm'' for coloring trapezoid graphs that improved the coloring algorithm in \cite{dagan1988trapezoid}. 

\begin{figure}[ht]
\begin{center}
	\fbox{
	    \parbox{11cm}{
		        \textbf{Input:} A sequence of boxes $\mathcal{B}=(B_{1},B_{2},\ldots, B_{n})$.\\
				\textbf{Output:} A partition $H$ of $\mathcal{B}$ into $k$-ary chains. \\
				
			\textbf{let}  $\mathcal{P}=\{(p_1, p_2) | (p_{1},p_{2}) = l_{B_i} $or$  (p_{1},p_{2}) = u_{B_i}, i \in 1 ... n\}$ \\

			\textbf{initialize} $S=\{d\}$, where $d$ is a real number smaller than all 
			\begin{itemize}  
			\item[] the $y$ coordinates of points $p \in \mathcal{P},$ marked \emph{available} 
			\end{itemize} 
			\textbf{foreach} $p \in \mathcal{P}$ sorted increasingly by the second coordinate \textbf{do}:
			\begin{itemize} 
			\item[] $q$ $\leftarrow$ first available slot below $p_{2}$ in $S$ 
			\item[] \textbf{if} $p=l(v)$ for some $v\in \mathcal{B}$ \textbf{then} 
			\begin{itemize} 
			\item[] \textbf{if} $q=u(w)_{2}$ for some $w\in \mathcal{B}$ \textbf{then}
			\begin{itemize} 
			\item[] insert $v$ in the $k$-ary chain of $w$ as a child of this node
			\item[] remove $q$ from $S$ 
			\item[] add $k$ copies of $p_{2}$ to $S$, marking them unavailable 
			\end{itemize} 
			\item[] \textbf{else} // $(q==d)$
			\begin{itemize} 
			\item[] start a new $k$-ary chain rooted at $v$
			\item[] add $k$ copies of $p_{2}$ to $S$, marking them unavailable 
			\end{itemize}
			\end{itemize} 
			\item[] \textbf{if} $p=u(v)$ for some $v\in \mathcal{B}$ \textbf{then} 
			\begin{itemize} 
			\item[] mark all slots with value $p_{2}$ in $S$  as available
			\end{itemize} 
			\end{itemize} 
			\textbf{return} the set of $k$-ary chains constructed by the algorithm. \\

	    }
	}
\end{center}
\caption{The greedy sweep-line algorithm for $x$-sorted trapezoid sequences.} 
\label{alg:boxes}
\end{figure}

As noted in Section~\ref{interpret}, partition into $k$-heapable sequences  is a natural generalization of coloring permutation graphs. 
In the sequel  we give an algorithm that generalizes the sweep-line coloring algorithm for trapezoid graphs from  \cite{dagan1988trapezoid} to the partition into $k$-ary chains of a particular class of sequences of boxes. 

\begin{theorem} 
Let $\mathcal{B}=(B_{1},B_{2},\ldots, B_{n})$ be a sequence of two-dimensional axis-parallel boxes, totally ordered by the $x$-coordinates of their right endpoints. Then the greedy sweep-line algorithm in Figure~\ref{alg:boxes} computes an optimal partition of sequence $\mathcal{B}$ into $k$-ary chains. 
\label{thm:boxes} 
\end{theorem} 

\begin{proof} 
The proof of Theorem~\ref{thm:boxes} leverages and extends the method used for intervals in the proof of Theorem~\ref{intervals}. The fact that sequence $\mathcal{B}$ is sorted in increasing order of the $x$-coordonate of the endpoints allows us to see the problem as one on intervals: instead of dealing with boxes $B_{1},B_{2},\ldots, B_{n}$ we will instead deal with the associated  intervals $I(B_{1}),I(B_{2}),\ldots, I(B_{n})$. 

We will assume that the $x$ and the $y$ coordinates of all boxes in $\mathcal{B}$ are all distinct. As usual with such algorithms, the truth of our statement does not rely on this assumption: if the statement is not true then points can be slightly perturbed to make the assumption true. The algorithm we give then extends to the degenerate cases as well. 

The sweep line maintains a multiset $S$ of \emph{slots}, that correspond to right endpoints of the intervals associated with the boxes processed so far. The difference with respect to the case of Theorem~\ref{intervals} is that the slots are now of one of two types: 
\begin{itemize} 
\item[-] \emph{unavailable: } a slot of this type is present in multiset $S$ but cannot be used to process intervals. 
\item[-] \emph{available: } a slot of this type can be fully employed when processing a new interval. 
\end{itemize} 

The action of a sweep line comprises two types of actions: 
\begin{itemize}
\item[-]\emph{insertion:} This happens when the sweep line reaches the left corner of some box $B_{j}$. We process the interval $I(B_{j})$ similarly to the process in Theorem~\ref{intervals}. Namely, we remove one lifeline from the largest \textbf{available slot} with value at most $l_{2}^{B_j}$, and insert 
into $S$ $k$ slots with value $r_{2}^{B_j}$. These slots are marked \textbf{unavailable.} 
\item[-]\emph{state change:} This happens when the sweep line reaches the right corner of some box $B_{j}$. All the slots with value $r_{2}^{B_j}$ change marking, from unavailable to \textbf{available. }
\end{itemize} 

The intuitive explanation for this modification is clear: a box $D$ can become the parent of another box $E$ only when the right corner of $D$ is to the left (on the $x$ axis) of the left corner of $E$. Thus, if the sweep line has not reached the right endpoint of a box $D$ then this box is not eligible to become the parent of any currently processed box. 

Let  $\sigma_{t}$ be the multiset of all slots created by the sweep-line algorithm up to a given moment $t$. Let $OPT_{t}$ be the multiset of slots corresponding to an \emph{optimal} partition into $k$-ary chains.  Let $\sigma^{'}_{t}$ be the corresponding (multi)set of all slots in $\sigma_{t}$ marked \emph{available} at moment $t$. Let $OPT_{t}^{'}$ be the set of slots in $OPT_{t}$ marked available at moment $t$.

The basis for the proof of Theorem~\ref{thm:boxes} is the following adaptation of Lemma~\ref{lemma:1} to the setting of slots with availability statuses: 

\begin{lemma}
For every $t\geq 0$, $sig(\sigma_{t}^\prime)\preceq sig(OPT_{t}^\prime)$. 
\label{lemma:1:boxes}
\end{lemma}
\begin{proof} 
By induction on $t$. The statement is clear for $t=0$, since both $\sigma_{0}$ and $OPT_{0}$ are empty. 
Assume the claim is true for all $t'<t$.  Let $sig(\sigma_{t-1}^{\prime})=[a_{1},a_{2},\ldots, a_{|sig(\sigma_{t-1}^{\prime})|}]$ and $sig(OPT_{t-1}^{\prime})=[b_{1},b_{2},\ldots, b_{|sig(OPT_{t-1}^{\prime})|}].$ Also, by convention, define $a_{0}=b_{0}=-\infty$ and \\ $a_{|sig(\sigma_{t-1}^{\prime})|+1}=b_{|sig(OPT_{t-1}^{\prime})|+1}=+\infty$. Finally, let $[l_{s}^{2},r_{s}^{2}]$ is the interval to be processed at stage $t$. 

Proving that $\sigma_{t-1}^{\prime} \preceq OPT_{t-1}^{\prime}$ entails proving that $|sig(\sigma_{t}^{\prime})|\leq |sig(OPT_{t}^{\prime})|$ and for all indices $1 \leq l \leq |sig(\sigma_{t}^{\prime})|$:
\begin{equation}
	sig(\sigma_{t}^{\prime})[l] \leq sig(OPT_{t}^{\prime})[l].
	\label{l}
\end{equation}
\begin{itemize} 
\item \textbf{Case 1: $t$ is an insertion step: }

$\sigma_{t}^\prime$  is obtained from $\sigma_{t-1}^\prime$ by removing the largest available slot less or equal to $l_{s}^{2}$. $OPT_{t}^\prime$ is obtained from $OPT_{t-1}^\prime$ by perhaps removing some slot with value at most $l_{s}^{2}$. 

The inequality $|\sigma_{t}^{\prime}|\leq |OPT_{t-1}^{\prime}|$ follows easilty from the corresponding inequality at stage $t-1$. Indeed, 
$\sigma_{t-1}^{\prime}$ may lose an element (in which case $OPT_{t-1}^{\prime}$ may also lose at most one element) or stay the same (in which case $OPT_{t-1}^{\prime}$ also stays the same (because there is no available slot to lose a lifeline.)

As for inequality~(\ref{l}): if $OPT_{t-1}^{\prime}$ does not lose an element the inequality follows easily from the induction hypothesis for stage $t-1$ and the fact that elements of $\sigma_{t-1}^{\prime}$ stay in place or shift to the left. 

If both $\sigma_{t-1}^{\prime}$ and $OPT_{t-1}^{\prime}$ lose one element to yield  $\sigma_{t}^{\prime}, OPT_{t}^{\prime}$, there are four types of positions $l$: 
\begin{itemize} 
\item Those below both deleted positions. They are unchanged as we move from $\sigma_{t-1}^{\prime}$ and $OPT_{t-1}^{\prime}$ to $\sigma_{t}^{\prime}, OPT_{t}^{\prime}$. That is: 
\[
\sigma_{t}^{\prime}[l]=\sigma_{t-1}^{\prime}[l]\mbox{ and } OPT_{t}^{\prime}[l]= OPT_{t-1}^{\prime}[l]. 
\]
Hence inequality~(\ref{l}) is true by the induction hypothesis. 
\item Those above both deleted positions. They get shifted to the left by one in both $\sigma_{t}^{\prime}$ and $OPT_{t}^{\prime}$. That is: 
\[
\sigma_{t}^{\prime}[l]=\sigma_{t-1}^{\prime}[l+1]\mbox{ and } OPT_{t}^{\prime}[l]= OPT_{t-1}^{\prime}[l+1]. 
\]

 Hence again inequality~(\ref{l}) is true by the induction hypothesis. 
\item Those below one but above the other deleted position. By the fact that the deleted slot in $OPT_{t-1}^{\prime}$ is less or equal 
than the one in $\sigma_{t-1}^{\prime}$ (since this is the largest available slot less or equal to $l_{s}^{2}$), it follows that 
\[
\sigma_{t}^{\prime}[l]=\sigma_{t-1}^{\prime}[l]\mbox{ and }OPT_{t}^{\prime}[l]=OPT_{t-1}^{\prime}[l+1]. 
\]
Hence 
\[
\sigma_{t}^{\prime}[l]=\sigma_{t-1}^{\prime}[l]\leq OPT_{t-1}^{\prime}[l]\leq OPT_{t-1}^{\prime}[l+1]= OPT_{t}^{\prime}[l]. 
\]
so relation~(\ref{l}) holds in all cases. 
\end{itemize}

\item \textbf{Case 2: $t$ is a change step: } The effect of such a step is that, both in $\sigma_{t-1}$ and $OPT_{t-1}$ the $k$ slots with value $r_{s}^{2}$ (that were previously inserted, but marked unavailable) become available. 

Since $|\sigma_{t}^{\prime}|=|\sigma_{t-1}^{\prime}|+k$ and, similarly, $|OPT_{t}^{\prime}|=|OPT_{t-1}^{\prime}|+k$, statement $|\sigma_{t}^{\prime}|\leq |OPT_{t}^{\prime}|$ follows from the analogous statement for $t-1$. 

Positions $l$ before the insertion points of the $k$ copies of $r_{s}^{2}$ are not modified in $\sigma_{t}^{\prime}, OPT_{t}^{\prime}$ so equation~(\ref{l}) follows for such $l$'s from the corresponding inequalities for stage $t-1$. Similarly, positions larger than both insertion points get shifted by exactly $k$, so inequality~(\ref{l}) also follows for such $l$'s by the corresponding inequality for stage $t-1$. 

These two cases cover all positions $l$ for which both values are different from the newly inserted values equal to $r_{s}^{2}$. If \emph{both} positions are equal to $r_{s}^{2}$ the inequality also follows. The only remaining cases are those $l$ for which one of 
$\sigma_{t}^{\prime}[l], OPT_{t}^{\prime}[l]$ is equal to the newly inserted value $r_{s}^{2}$ but the other is not. On inspection, though, inequality~(\ref{l}) is true in these cases as well: because of dominance, the insertion point into $OPT_{t-1}^{\prime}$ is to the left (or equal) to the insertion point into $\sigma_{t-1}^{\prime}$. So if $\sigma_{t}^{\prime}[l]\neq r_{s}^{2}$ but $OPT_{t}^{\prime}[l]=r_{s}^{2}$
then $\sigma_{t}^{\prime}[l]< r_{s}^{2}= OPT_{t}^{\prime}[l]$. The reasoning in the case $\sigma_{t}^{\prime}[l]= r_{s}^{2}$ but $OPT_{t}^{\prime}[l]\neq r_{s}^{2}$ is analogous, the conclusion being that $\sigma_{t}^{\prime}[l]= r_{s}^{2}< OPT_{t}^{\prime}[l]$.
\end{itemize} 
\end{proof} 

Using Lemma~\ref{lemma:1:boxes} we infer that at every step $t$ when the greedy sweep-line algorithm creates a new $k$-ary chain (because there is no available slot to lose a lifeline), so does the optimal algorithm (for the very same reason). Hence the greedy sweep-line algorithm is optimal. 
\end{proof} 

\section{Maximal heapable subsets of interval orders}
\label{maxheap} 
Finally, note that we are unable to solve the open problem from \cite{byers2011heapable} on the complexity of computing a maximum heapable subsequence, i.e. a maximum treelike $2$-independent set in permutation posets (see the definition and discussion in Section \ref{interpret}). Instead, we settle this problem for a different subclass of partial orders, the class of interval orders: 

\begin{figure}
\begin{center}
	\fbox{
	    \parbox{13cm}{
		        \textbf{Input:} A set of intervals $I$.\\
				\textbf{Output:} A $k$-ary chain  $J\subseteq I$ of maximum cardinality. \\
				
{\bf Sort} the intervals w.r.t. $\sqsubseteq$: $SI=(I_{1},\ldots, I_{n})$. \\		
{\bf Let} $J=\emptyset$. \\		
			\textbf{For} $i:=1$ to $n$ do:
			\begin{itemize}
				\item[] \textbf{If} $I_{i}=[l_{i},r_{i}] $ can be inserted into some empty slot   							
				\begin{itemize}
					\item[] \textbf{then} 
					\item[] \hspace{1cm} $J=J\cup \{i\}$. 
					\item[] \hspace{1cm} insert $I_i$ in the highest-valued compatible slot.				
				\end{itemize}
			\end{itemize}
	    }
	}
\end{center}
\caption{The greedy best-fit algorithm for sets of intervals.} 
\label{alg:greedy3}
\end{figure} 

\begin{theorem} Given a set of intervals $I$, the best-fit algorithm in Figure~\ref{alg:greedy3} computes a largest $k$-heapable subset of intervals of the set $I$. 
\label{max}
\end{theorem} 
\begin{proof} 

Note first that, whereas the theorem deals with \emph{subsets} of the set of intervals, the algorithm in Figure~\ref{alg:greedy3} considers the intervals in the set in a specific order (by sorting them with respect to $\sqsubseteq$). This will turn not to be a problem, because we prove by induction the following result, which implies the statement of the theorem: 

\begin{lemma} For every $1\leq r\leq n$ there exists a largest $k$-heapable subset $\Gamma_{r}$ of $\{I_{1},I_{2}\ldots, I_{n}\}$ and a $k$-heap $T_{r}$ for $\Gamma_{r}$ such that the tree $T_{G}$ built by the greedy algorithm of Figure~\ref{alg:greedy3} agrees with $T_{r}$ with respect to the presence or absence of intervals $I_{1},I_{2}\ldots, I_{r}.$ 
\end{lemma} 
\begin{proof} By induction. 

The result is simple for $r=1$: if $I_{1}$ is part of an optimal $k$-heapable subset then there is nothing to prove, since $I_{1}$ can only participate in a $k$-heapable subset as the root of its corresponding tree. This happens because  $I_{1}\sqsubseteq I_{s}$ for all $s\geq 1$. If, on the other hand, $I_{1}$ is not part of an optimal subset, then consider such a solution $\Gamma$. Create a new solution $\Gamma_{1}$ with the same tree shape by replacing the root interval $I$ of $\Gamma$ by $I_{1}$ (thus obtaining tree $T_{1}$). This is legal, since $I_{1}$ ends earlier than $I$. 

Assume we have proved the result for $1\leq i\leq r-1$. Consider now the optimal solution $\Gamma_{r-1}$ (with $k$-ary tree $T_{r-1}$) agreeing with the greedy solution $T_{G}$ on the presence of the first $r-1$ intervals. 
\begin{itemize}
\item If $I_{r}$ cannot be inserted into $T_G$ (that is, in the portion of $T_{G}$ constructed by the greedy-best fit algorithm after considering intervals $I_1, I_{2},\ldots, I_{r-1}$) it means that no slot is available for $I_{r}$. By domination (Lemma~\ref{lemma:1}) this must be true for both the greedy and the optimal solution $\Gamma_{r-1}$. Therefore these two solutions do not contain $I_{r}$, hence they agree on the presence or absence of $I_{1},I_{2},\ldots, I_{r}$. 
\item Suppose now that $I_{r}$ can be inserted in a slot (thus is present in the greedy solution) but is {\bf not} present in the optimal solution $\Gamma_{r}$. Let $J$ be the interval that has the same parent in $\Gamma_{r}$ as $I_{r}$ has in the greedy solution (Figure~\ref{first-flow}). $J$ must exist, otherwise one could simply extend $\Gamma_{r-1}$ by inserting $I_{r}$ at that position. 

By the order we considered, the intervals $I_{r}$ ends no later than $J$ does. So by inserting $I_{r}$ instead of $J$ we obtain an optimal solution $\Gamma_{r}=\Gamma_{r-1}\setminus \{J\}\cup \{I_{r}\}
$ that satisfies the induction property. 
\end{itemize}  
\end{proof} 

The conclusion of the induction argument is that there exists an optimal tree containing \emph{exactly the same} intervals as those selected by $T_{G}$. Thus the greedy best-fit algorithm produces an optimal solution. 
\end{proof}

\begin{figure}[ht]
  \begin{center}
  
    \begin{minipage}{.25\textwidth} 
      
		\begin{tikzpicture}[scale=0.6]
		\coordinate [label=left:$$] (A) at (-2cm,-1.cm);
		\coordinate [label=right:$$] (C) at (2cm,-1.0cm);
		\coordinate [label=above:$\Gamma_{r-1}$] (B) at (1cm,2.0cm);
		\coordinate [label=above:$y$] (X) at (1cm,-0.9cm);

		\coordinate [label=above:$$] (F) at (-0.5cm,-0.9cm);
		\coordinate [label=above:$$] (F1) at (-0.5cm,-1.1cm);
		\coordinate [label=above:$$] (F2) at (-0.5cm,-0.8cm);

		\coordinate [label=left:$J$] (J) at (1cm,-2cm);
		\coordinate [label=above:$$] (J1) at (1.5cm,-4.0cm);
		\coordinate [label=above:$$] (J2) at (0.5cm,-4.0cm);
		
		\coordinate [label=above:$$] (L1) at (-1.5cm,-4.0cm);
		\coordinate [label=above:$$] (L2) at (-0.5cm,-4.0cm);
		
		\draw [->] (X) -- (J);
		\draw (A) -- node[sloped,above] {} (B) -- node[sloped,above,] {} (C) -- node[below] {} (A);
		\draw[dashed] (J) -- node[sloped,above] {} (J1) -- node[sloped,above,] {} (J2) -- node[below] {} (J);
		
		\end{tikzpicture}      
      
   \end{minipage}
   \begin{minipage}{.25\textwidth} 

		\begin{tikzpicture}[scale=0.6]
		\coordinate [label=left:$$] (A) at (-1cm,-1.cm);
		\coordinate [label=right:$$] (C) at (2cm,-1.0cm);
		\coordinate [label=above:$T_{G}$] (B) at (-1.0cm,2.0cm);
		\coordinate [label=above:$y$] (Y) at (1cm,-0.8cm);		
		
		\coordinate [label=left:$I_{r}$] (J) at (1cm,-2cm);
		\coordinate [label=above:$$] (J1) at (1.5cm,-4.0cm);
		\coordinate [label=above:$$] (J2) at (0.5cm,-4.0cm);

		\draw [->] (Y) -- (J);
		\draw (A) -- node[sloped,above] {} (B) -- node[sloped,above,] {} (C) -- node[below] {} (A);
		\draw[dashed] (J) -- node[sloped,above] {} (J1) -- node[sloped,above,] {} (J2) -- node[below] {} (J);

		\end{tikzpicture}

     \end{minipage} $\Rightarrow$
    \begin{minipage}{.25\textwidth} 
      
		\begin{tikzpicture}[scale=0.6]
		\coordinate [label=left:$$] (A) at (-2cm,-1.cm);
		\coordinate [label=right:$$] (C) at (2cm,-1.0cm);
		\coordinate [label=above:$\Gamma_{r}$] (B) at (1cm,2.0cm);
		\coordinate [label=above:$y$] (X) at (1cm,-0.9cm);

		\coordinate [label=above:$$] (F1) at (-0.5cm,-1.1cm);
		\coordinate [label=above:$$] (F2) at (-0.5cm,-0.8cm);

		\coordinate [label=left:$I_{r}$] (J) at (1cm,-2cm);
		\coordinate [label=above:$$] (J1) at (1.5cm,-4.0cm);
		\coordinate [label=above:$$] (J2) at (0.5cm,-4.0cm);
		
		\coordinate [label=above:$$] (L1) at (-1cm,-4.0cm);
		\coordinate [label=above:$$] (L2) at (0cm,-4.0cm);
		
		\draw [->] (X) -- (J);
		\draw (A) -- node[sloped,above] {} (B) -- node[sloped,above,] {} (C) -- node[below] {} (A);
		\draw[dashed] (J) -- node[sloped,above] {} (J1) -- node[sloped,above,] {} (J2) -- node[below] {} (J);
		
		\end{tikzpicture}      
      
   \end{minipage}

  \end{center}
\caption{Creating optimal solution $\Gamma_{r}=\Gamma_{r-1}\setminus \{J\}\cup \{I_{r}\}
$. }
\label{first-flow} 
\end{figure}

 \section{Open questions and future work}
\label{sec:open}

Our Theorem \ref{main:thm} is very similar to (the proof of) Dilworth's theorem. It is not {\it yet} a proper generalization of this result  to the case $k\geq 1$ because of the lack of a suitable extension of the notion of {\it antichain}: 

\begin{open} 
Is there a suitable definition of the concept of $k$-antichain, that coincides with this concept for $k=1$ and leads (via our theorem \ref{main:thm}) to an extension of Dilworth's theorem?
\end{open} 

On the other hand, results in Section \ref{sec:interval} naturally raise the following:

\begin{open}
For which partial orders $Q$ can one compute the parameter $k$-wd(Q) (and an associate optimal $k$-ary chain decomposition) via a greedy algorithm? 
\end{open} 
 
Several open problems concern the limit behavior of the expected value of the $k$-width of a set of random intervals, for $k\geq 1$.  As discussed in Section \ref{sec:prelim}, for $k=1$ the scaling behavior of this parameter is known (\cite{justicz1990random}). However, in the case of random permutations, the most illuminating description of this scaling behavior is by analyzing the hydrodynamic limit of the Hammersley process (\cite{aldous1995hammersley,groeneboom2002hydrodynamical}). As shown in Section \ref{sec:sets}, the difference between sequences and sets of intervals is not substantial. 

We ask, therefore, whether the success in analyzing this process for random permutations can be replicated in the case of sequences/sets of random intervals: 

\begin{open} 
Analyze the hydrodynamic limit of the Hammersley process for sequences/sets of random intervals. 
\end{open}

When $k\geq 2$ even the scaling behavior is not known, for both sequences and sets of random intervals. The connection with the interval Hammersley process given by Corollaries ~\ref{cor:1} and ~\ref{cor:2} provides a convenient, lean way to simulate the dynamics, leading to experimental observations on the scaling constants. A C++ program used to perform these experiments is publicly available at \cite{heapable-interval}.  Based on these experiments we would like to raise the following: 

\begin{conjecture} For every $k\geq 2$ there exists a positive constant $c_{k}>0$ such that, if $R_{n}$ is a sequence of $n$ random intervals then 
\begin{equation} 
\lim_{n\rightarrow \infty} \frac{E[k\mbox{-wd}(R_n)]}{n}=c_{k}
\end{equation} 
Moreover $c_{k}=\frac{1}{k+1}$.
\end{conjecture} 

According to this conjecture, just as in the case of random permutations, the scaling behavior of the $k$-width changes when going from $k=1$ to $k=2$. Note, though, that the direction of change is different ($\Theta(\sqrt{n})$ to $\Theta(\log{n})$ for integer sequences, $\Theta(\sqrt{n})$ to $\Theta(n)$ for sequences of intervals). 
On the other hand, somewhat surprisingly, the scaling behavior of sets of random intervals seems to be similar to that for sequences: 

\begin{conjecture} For every $k\geq 2$ there exists a positive constant $d_{k}>0$ such that, if $W_{n}$ is a {\bf set} of $n$ random intervals then 
\begin{equation} 
\lim_{n\rightarrow \infty} \frac{E[k\mbox{-wd}(W_n)]}{n}=d_{k}
\end{equation} 
Experiments suggest that $d_{2}=c_{2}=\frac{1}{3}$, and similarly for $k=3,4$ $d_{k}=c_{k}=\frac{1}{k+1}$. Therefore we conjecture  that 
\begin{equation}
d_k=c_k=\frac{1}{k+1}\mbox{ for all }k\geq 2. 
\end{equation} 
\end{conjecture}

\bibliographystyle{abbrvnat} 
\bibliography{heapability-dilworth}

\end{document}